\documentclass[12pt]{amsart}
\usepackage{amsmath,amsfonts,amssymb,amsthm}
\usepackage{anysize}
\marginsize{2cm}{2cm}{2cm}{2cm}
\usepackage{graphicx}
\usepackage{enumerate}

\def\z{\mathfrak{z}}
\def\u{\mathfrak{u}}

\def\g{\mathfrak{g}}
\def\h{\mathfrak{h}}
\def\n{\mathfrak{n}}
\def\v{\mathfrak{v}}
\def\a{\mathfrak{a}}

\def\C{\mathbb{C}}
\def\R{\mathbb{R}}

\def\G{\mathbb G}

\def\ad{\operatorname{ad}}

\def\alt{\raise1pt\hbox{$\bigwedge$}}
\def\pint{\langle \cdotp,\cdotp \rangle }

\theoremstyle{plain}
\newtheorem{teo}{\bf Theorem}[section]
\newtheorem{cor}[teo]{\bf Corollary}
\newtheorem{prop}[teo]{\bf Proposition}
\newtheorem{lema}[teo]{\bf Lemma}

\theoremstyle{definition}

\newtheorem{example}{Example}

\theoremstyle{remark}
\newtheorem{rem}[teo]{\bf Remark}

\newcommand{\ri}{{\rm (i)}}
\newcommand{\rii}{{\rm (ii)}}

\title{Killing-Yano $2$-forms on 2-step nilpotent Lie groups}

\author{Adri\'an Andrada}
\email{andrada@famaf.unc.edu.ar}

\author{Isabel G. Dotti}
\email{idotti@famaf.unc.edu.ar}

\date{}
\address{FaMAF-CIEM, Universidad Nacional de C\'ordoba, Ciudad Universitaria, 5000 C\'ordoba, Argentina}
\thanks{Both authors were partially supported by CONICET, ANPCYT and SECYT-UNC (Argentina).}

\dedicatory{}

\begin{document}

\begin{abstract}
In this article we show that the only 2-step nilpotent Lie groups which carry a non-degenerate left invariant Killing-Yano 2-form are the complex Lie groups. In the case of 2-step nilpotent complex Lie groups arising from connected graphs, we prove that the space of left invariant Killing-Yano 2-forms is one-dimensional.
\end{abstract}

\maketitle

\section{Introduction}

A $p$-form $\omega$ on a Riemannian manifold $(M,g)$ is called \textit{Killing-Yano} if it satisfies the \textit{Killing-Yano equation},
\begin{equation}  
(\nabla_X \omega)(Y,Z_1,\ldots,Z_{p-1}) = - (\nabla_Y \omega)(X,Z_1,\ldots,Z_{p-1}), 
\end{equation}
where $\nabla$ is the Levi-Civita connection and $X,Y,Z_1,\ldots,Z_{p-1}$ are arbitrary vector fields on $M$. These forms were introduced by K. Yano in \cite{Y} and they can be considered as generalizations of Killing vector fields. Indeed, a $1$-form is Killing-Yano if and only if its dual vector field is Killing. 

Clearly, any parallel $p$-form is Killing-Yano. However, it is not easy to find examples of non-parallel Killing-Yano $p$-forms, as the following results show:
\begin{itemize}
 \item If $(M^7,g)$ is a compact manifold with holonomy $G_2$ then any Killing-Yano form is parallel \cite{Sem}. 
 \item Every Killing-Yano $p$-form on a compact quaternion-K\"ahler manifold is parallel for any $p\geq 2$ \cite{MS}.
 \item A compact simply connected symmetric space carries a non-parallel Killing-Yano $p$-form ($p\geq 2$) if and only if it is isometric to a Riemannian product $S^k\times N$, where $S^k$ is a round sphere and $k > p$ \cite{BMS}.
\end{itemize}

A natural source of examples in Riemannian geometry is provided by Lie groups equipped with left invariant Riemannian metrics. The study of left invariant Killing-Yano 2-forms on Lie groups with left invariant metrics started in \cite{BDS} and was continued in \cite{ABD,AD}. In particular, it was proved in \cite{AD} that there are no left invariant parallel non-degenerate 2-forms on nilpotent Lie groups. Therefore, if we are able to find non-degenerate Killing-Yano 2-forms on a nilpotent Lie group, then they will be automatically non-parallel. In this direction, the only known examples of left invariant Killing-Yano 2-forms on nilpotent Lie groups appear on \textit{2-step nilpotent} Lie groups, that is, when the commutator ideal is contained in the center of the group (see \cite[Theorem 5.1]{BDS}). In this article we refine this result by showing that non-degenerate Killing-Yano 2-forms appear only on \textit{complex} 2-step nilpotent Lie groups endowed with a left invariant metric (Theorem \ref{main2}) and, furthermore, any such Lie group admits a left invariant metric with non-degenerate Killing-Yano 2-forms (Proposition \ref{ida}). In Section \ref{dim=1} we exhibit a family of examples where the dimension of the space of left invariant Killing-Yano 2-forms is 1.

\

\section{Preliminaries}
A $2$-form $\omega$ on a Riemannian manifold $(M,g)$ is called \textit{Killing-Yano} if it satisfies the \textit{Killing-Yano equation},
\begin{equation}\label{yano}  
(\nabla_X \omega)(Y,Z) = - (\nabla_Y \omega)(X,Z), 
\end{equation}
where $\nabla$ is the Levi-Civita connection and $X,Y,Z$ are arbitrary vector fields on $M$ (see \cite{Y}).  The following result is clear.

\begin{lema}\label{equiv}
Let $(M,g)$ be a Riemannian manifold, $\nabla$ the Levi-Civita connection and $\omega$ a $2$-form on $M$. The following conditions are equivalent:
	\begin{enumerate}
		\item[$\ri$] $ ( \nabla_X \omega)(Y,Z) + (\nabla_Y \omega)(X,Z)=0$;
		\item[$\rii$] $ d \omega(X,Y,Z) = 3(\nabla_X\omega) (Y,Z).$
	\end{enumerate}
\end{lema}

\

Let $(M,g)$ be a Riemannian manifold, $\nabla$ the Levi-Civita connection and 
$\omega$ a $2$-form on $M$ satisfying any of the conditions of Lemma 
\ref{equiv}. Then the skew-symmetric endomorphism $T$ of $TM$ defined by 
$\omega$ and $g$, that is $\omega(X,Y)=g(TX,Y)$ satisfies 
\begin{equation}\label{yano_T}  
(\nabla_X T)Y = - (\nabla_Y T)X
\end{equation} 
for all $X,Y$   vector fields on $M$. Conversely, if $T$ is a skew-symmetric 
endomorphism of $TM$ satisfying \eqref{yano_T} then the 2-form $\omega(X,Y)= 
g(TX,Y)$ satisfies any of the conditions of Lemma \ref{equiv}.  After this 
observation we will refer equivalently to a 2-form $\omega$ satisfying \eqref{yano} or a skew-symmetric 
$(1,1)$-tensor $T$ satisfying or \eqref{yano_T} as Killing-Yano (KY). Note that if $(J,g)$ is an almost Hermitian structure, then the 
fundamental $2$-form $\omega$ given by $\omega(X,Y)=g(JX,Y)$ is Killing-Yano if and only if $(J,g)$ is nearly K\"ahler.

\

Let $G$ be an $n$-dimensional Lie group and let $\g$ be the associated Lie 
algebra of all left invariant vector fields on $G$. If $T_eG$ is the  tangent 
space of $G$ at $e$, the identity of $G$, the correspondence $X \to X_e:=x$ from 
$\mathfrak g \to T_eG$ is a linear isomorphism. This isomorphism allows to 
define a Lie algebra structure on  the tangent space $T_eG$  setting, for $x, y 
\in T_eG $,  $[x,y] = [X, Y]_e$ where $X,Y$ are the left invariant vector fields 
defined by $x,y$, respectively.

A left invariant metric on $G$ is a Riemannian metric such that $L_a$, the left 
multiplication by $a\in G$, is an isometry for every $a \in G$. Conversely, 
every inner product on $T_eG$ gives rise, by left translations, to a left 
invariant metric. A Lie group equipped with a left invariant metric is therefore 
a homogeneous Riemannian manifold where many geometric invariants can be 
computed at the Lie algebra level. In particular, the Levi-Civita connection 
$\nabla$ associated to a left invariant metric $g$, when applied to left 
invariant vector fields, is given by:
\begin{equation}\label{LC}
 2\langle\nabla_x y,z \rangle = \langle [x, y],z \rangle -  
 \langle [y,z],x \rangle +  \langle [z,x],y \rangle, \qquad x,y,z \in \mathfrak g,
\end{equation}
where $\langle \cdot,\cdot \rangle$ is the inner product induced by $g$ on 
$\g$. Note that $\nabla g =0$ is equivalent to 
$\nabla_x\in\mathfrak{so}(\g,\langle \cdot,\cdot \rangle)$ for any 
$x\in\mathfrak g$. 

\

A left invariant $p$-form $\omega$ on $G$ is a $p$-form such that $L_a^*\omega 
= \omega$ for all $a\in G$. We will consider left invariant $2$-forms $\omega$ 
on $(G,g)$ satisfying \eqref{yano}. Since $\nabla \omega$ and $d\omega$ are left 
invariant as well, we will study $\omega \in \alt^2 \g^*$ satisfying 
\eqref{yano} for $x,y,z \in \g$. We will say then that $\omega$ is a 
Killing-Yano (KY) $2$-form on $\g$. Clearly, the corresponding skew-symmetric 
$(1,1)$-tensor $T$ defined by $\omega$ and $g$ is left invariant, so it is 
determined by its value in $e\in G$, which will be denoted by $T:\g\to\g$. The 
endomorphism $T$ will be called a Killing-Yano (KY) tensor on $\g$.
              
\

\begin{lema}\label{invariante}
Let $T$ be a KY tensor on a Lie algebra $\g$ with inner product $\langle 
\cdot,\cdot \rangle$. If $\u$ is a $T$-invariant Lie subalgebra of $\g$ then 
$T|_\u$ is a KY tensor on $\u$.
\end{lema}

\begin{proof}
Let $\nabla^\u$ denote the Levi-Civita connection on $\u$ associated to the 
restriction of $\langle \cdot,\cdot \rangle$ to $\u$, and let $\nabla^\g$ denote 
the Levi-Civita connection on $\g$ associated to $\langle \cdot,\cdot \rangle$. 
Then it follows easily from \eqref{LC} that
\[ \langle (\nabla^\u_x T|_\u)x,y\rangle = \langle (\nabla^\g_x T)x,y\rangle, 
\quad x,y\in\u.   \]
Therefore if $T$ is KY on $\g$ then $T|_\u$ is KY on $\u$.
\end{proof}           

\

We recall also the following result proved in \cite{AD}, which holds in 
particular for non abelian nilpotent Lie groups.

\begin{teo}\label{nilp}\cite[Theorem 5.1]{AD}
If $\g$ is a Lie algebra with an inner product $\pint$ such that $\g'\cap\z\neq 
\{0\}$, then there is no skew-symmetric invertible parallel tensor on $\g$. In 
particular, this holds for any inner product on a non-abelian nilpotent Lie 
algebra. 
\end{teo}

\

\section{Main results}

We will consider next the case of a $2$-step nilpotent Lie algebra $\n$ 
equipped with an inner product $\pint$. We recall that a Lie algebra $\n$ is 
called $2$-step nilpotent if it is not abelian and $[[\n,\n],\n]=0$, or 
equivalently, the commutator ideal $\n':=[\n,\n]$ is contained in the center 
$\z$ of $\n$. Such a Lie algebra is unimodular and its associated simply 
connected Lie group is diffeomorphic to a Euclidean space $\R^n$ via the 
exponential map.

If $\n$ is a $2$-step nilpotent Lie algebra equipped with an inner product 
$\pint$, then there is an orthogonal decomposition $\n=\z\oplus\v$, where $\z$ 
is the center of $\n$ and $\v$ is its orthogonal complement.

\

In \cite{BDS} a characterization of $2$-step nilpotent Lie groups with left 
invariant metrics admitting a left invariant Killing-Yano tensor was obtained. 
Namely, the following result was proved:
 
\begin{teo}\label{bds}\cite[Theorem 3.1]{BDS} 
Let $T$ be a skew-symmetric endomorphism of $\n= \z \oplus \v$. Then $T$ is a 
Killing-Yano tensor on $\n$ if and only if 
\begin{equation}\label{KY}
T(\z)\subseteq \z, \quad \text{hence also} \quad T(\v)\subseteq \v;
\end{equation}
and 
\begin{equation}\label{KY1}
T[x,y]=3[Tx,y] \quad\text{for all } x,y\in\v.
\end{equation}
\end{teo}

\medskip

Note that \eqref{KY1} implies $[Tx,y]=[x,Ty]$ for any $x,y\in\v$.

\medskip

\begin{cor}
There are no nearly K\"ahler structures on 2-step nilpotent Lie algebras.
\end{cor}

\begin{proof}
If $(J,g)$ is a nearly K\"ahler structure on a 2-step nilpotent Lie algebra 
$\n$, then $J$ is a KY tensor on $\n$. Therefore, for $x,y\in\v$,
\[ [x,y]=-J^2[x,y]=-3J[Jx,y]=-9[J^2x,y]=9[x,y]\]
and therefore $[x,y]=0$, so that $\n$ would be abelian.
\end{proof}

\

We show next that the existence of a KY tensor on a 2-step nilpotent Lie algebra $\n$ imposes strong restrictions on $\n$. Indeed,

\begin{teo}\label{main1}
Let $\n$ be a $2$-step nilpotent Lie algebra with an inner product $\pint$. If 
$T:\n\to\n$ is a KY tensor on $\n$ then:
\begin{enumerate}
\item[$(a)$] $\n$ is isometrically isomorphic to a direct product of ideals 
$\n=\n_1\times \n_2$, where $T|_{\n_1}=0$, $\n_2$ is $T$-invariant and 
$T|_{\n_2}$ is an invertible KY tensor on $\n_2$;
\item[$(b)$] $T$ is parallel if and only if $\n_2$ is abelian. Moreover, if 
$\z=\n'$ then $T=0$.
\end{enumerate}
\end{teo}

\begin{proof}
Since $T$ is skew-symmetric we have an orthogonal decomposition $\n=\ker 
T\oplus \operatorname{Im} T$ into $T$-invariant subspaces. Recall that we also 
have another orthogonal decomposition $\n=\z\oplus \v$, where $\z$ is the center 
of $\n$. According to Theorem \ref{bds} both $\z$ and $\v$ are $T$-invariant, 
therefore, if $x\in\ker T$ is decomposed as $x=x_1+x_2$ with $x_1\in\z$ and 
$x_2\in\v$, then $0=Tx=Tx_1+Tx_2$ with $Tx_1\in\z$ and $Tx_2\in\v$. Hence, 
$Tx_1=0$ and $Tx_2=0$, so that $\ker T=(\ker T\cap \z)\oplus (\ker T\cap \v)$.  
In the same way it is shown that $\operatorname{Im} T=(\operatorname{Im} T\cap 
\z)\oplus (\operatorname{Im} T\cap \v)$.

We will show next that both $\ker T$ and $\operatorname{Im} T$ are ideals of 
$\n$. Indeed:
\begin{itemize}
\item if $x\in\ker T,\,y\in\n$, then $x=x_1+x_2$, $y=y_1+y_2$ with $x_1\in\ker 
T\cap \z$, $x_2\in \ker T\cap\v$ and $y_1\in\z$, $y_2\in\v$. Hence, 
$T[x,y]=T[x_2,y_2]=3[Tx_2,y_2]=0$, using \eqref{KY1}. Thus $\ker T$ is an ideal 
of $\n$.
\item if $Tx\in \operatorname{Im} T$ and $y\in\n$, then $x=x_1+x_2$ and 
$y=y_1+y_2$ with $x_1,y_1\in \z$ and $x_2,y_2\in \v$. Then we have 
$[Tx,y]=[Tx_2,y_2]=\frac13 T[x_2,y_2]\in \operatorname{Im} T$, thus 
$\operatorname{Im} T$ is an ideal of $\n$.
\end{itemize}
Therefore, setting $\n_1:=\ker T$ and $\n_2:=\operatorname{Im} T$ we obtain 
$\n=\n_1\times \n_2$. Clearly, $T|_{\n_1}=0$, $\n_2$ is $T$-invariant and 
$T|_{\n_2}$ is invertible. Moreover, it follows from Lemma \ref{invariante} that 
$T|_{\n_2}$ is a KY tensor on $\n_2$. This proves $(a)$.

\medskip

It is clear that $\n_2$ is either abelian or $2$-step nilpotent. If $\n_2$ is 
$2$-step nilpotent, it follows from Theorem \ref{nilp} that $T|_{\n_2}$ is not 
parallel, since it is invertible. Now, if $\n_2=\operatorname{Im}T$ is abelian, 
then this means that $\operatorname{Im}T\subseteq\z$. Hence $T|_\v=0$ and it 
follows from \eqref{KY1} that $T|_{\n'}=0$. It is easy to verify that these 
conditions imply that $T$ is parallel, and $T=0$ if $\n'=\z$. Therefore $(b)$ 
follows.
\end{proof}

\medskip

\begin{rem}
Note that when the KY tensor $T$ is invertible the KY 2-form $\omega$ 
associated to $T$, $\omega(X,Y)=g(TX,Y)$, is non-degenerate.
\end{rem}

\medskip

As a consequence of Theorem \ref{main1}, we may restrict ourselves to the study 
of \textit{invertible} KY tensors on 2-step nilpotent Lie groups. A large number 
of such examples appear in complex 2-step nilpotent Lie groups, as the following 
result shows:

\begin{prop}\label{ida}
Let $N$ be a complex 2-step nilpotent Lie group equipped with a left invariant Hermitian metric $g$. Then $(N,g)$ admits a non-parallel invertible KY tensor.
\end{prop}

\begin{proof}
Let $J$ denote the associated bi-invariant complex structure on $N$. Therefore 
the restriction of $J$ to $\n$, the Lie algebra of $N$, satisfies: 
$J:\n\to\n$, $J^2=-\operatorname{Id}$ and $J\ad_x=\ad_xJ$ for any $x\in\n$, 
i.e., $J[x,y]=[x,Jy]=[Jx,y]$ for any $x,y\in\n$. 

Let $\pint:=g|_\n$ denote the induced inner product on $\n$. If $\z$ denotes 
the center of $\n$, let $\v:=\z^\perp$, so that $\n=\z\oplus\v$, orthogonal sum. 
Note that both $\z$ and $\v$ are $J$-invariant, as well as the commutator ideal 
$\n'=[\n,\n]$.

Let us define a skew-symmetric isomorphism $T$ of $\n$ in the following way: if 
we decompose $\z=\n'\oplus\a$ as an orthogonal sum, we set
\[ T|_{\v}=J|_{\v}, \quad T|_{\n'}=3J|_{\n'}, \quad T|_{\a}=T_0,  \]
where $T_0:\a\to\a$ is any skew-symmetric isomorphism of $\a$. It follows from 
Theorem \ref{bds} that $T$ is a left invariant Killing-Yano tensor on $N$. Moreover, according to Theorem \ref{nilp}, $T$ is not parallel.
 \end{proof}

\begin{rem}
We point out that the classification of complex 2-step nilpotent Lie algebras, up to isomorphism, is not known; furthermore, this is considered to be a wild problem in the literature (see for instance \cite{DLV}, where several partial results are collected).
\end{rem}

\

The next theorem will show that the complex 2-step nilpotent Lie groups exhaust 
the family of 2-step nilpotent Lie groups with left invariant metric admitting 
invertible KY tensors.

\begin{teo}\label{main2}
If a 2-step nilpotent Lie group $N$ equipped with a left invariant metric $g$ admits an invertible left invariant KY tensor then $N$ is a complex Lie group. Moreover, $g$ is Hermitian with respect to this complex structure.
\end{teo}

\begin{proof}

Let $N$ be a $2$-step nilpotent Lie group equipped with a left invariant metric 
$g$ and a left invariant Killing-Yano tensor $T$. We denote by $\pint:=g|_\n$ 
the induced inner product on $\n$, and we also denote by $T$ the restriction 
$T|_\n$ of $T$ to $\n$. Therefore, according to Theorem \ref{bds}, $T$ preserves 
the center $\z$ of $\n$ and \eqref{KY1} holds.

Let us consider the orthogonal decomposition $\n=\a\oplus\n'\oplus \v$, where 
$\z=\a\oplus\n'$ and $\v=\z^\perp$. Note that due to \eqref{KY1}, $T$ preserves 
each of these subspaces, since $T$ is skew-symmetric. Therefore, there exists an 
orthonormal basis $\{e_1,\ldots,e_n,f_1,\ldots,f_n\}$ of $\v$ such that 
\[Te_i=a_i f_i,\quad Tf_i=-a_i e_i, \quad i=1,\ldots,n,\]
for some non-zero $a_i\in\R$. Moreover, interchanging $e_i$ with $f_i$ if 
necessary, we may assume that $a_i>0$ for all $i$ and, after a further 
reordering of the basis, we may assume also that $0<a_1\leq a_2\leq\cdots\leq a_n$. 
Let us denote $\v_i=\text{span}\{e_i,f_i\}$, so that $\v=\oplus_i\v_i$.

It follows from \eqref{KY1} that $[\v_i,\v_i]=0$. Indeed, 
$T[e_i,f_i]=3[a_if_i,f_i]=0$, hence $[e_i,f_i]=0$. In particular, for each $i$ 
there exists $j$ such that $[\v_i,\v_j]\neq 0$, since otherwise $\v_i\subset\z$. 
Let $(i,j)$ be a pair of indices such that $[\v_i,\v_j]\neq 0$. Then at least 
one of $[e_i,e_j]$, $[e_i,f_j]$, $[f_i,e_j]$ and $[f_i,f_j]$ is non-zero. Then, 
according to \eqref{KY1}, we have that
\begin{gather*}
T[e_i,e_j] = 3a_i[f_i,e_j]=3a_j[e_i,f_j], \\
T[f_i,f_j]=-3a_i[e_i,f_j]=-3a_j[f_i,e_j].
\end{gather*}
Since $T$ is invertible, it follows that all four Lie brackets are non-zero. 
Furthermore, we have that 
\[ [e_i,f_j]=\frac{a_i}{a_j}[f_i,e_j],\qquad [e_i,f_j]=\frac{a_j}{a_i}[f_i,e_j].\]
As the constants are positive, we obtain $a_i=a_j$. Hence, 
$[e_i,e_j]=-[f_i,f_j]$, $[e_i,f_j]=[f_i,e_j]$ and 
$[\v_i,\v_j]=\text{span}\{[e_i,e_j], [e_i,f_j]\}$, with $\dim [\v_i,\v_j]=2$, 
and $[\v_i,\v_j]$ is $T$-invariant.

Let us define a series of integers $\{i_j:j=1,\ldots,r\}$ in the following 
way: 
\[ i_1=1, \quad i_j=\min\{k\mid a_k\neq a_{i_{j-1}}\}  \]
We define now the following subspaces of $\v$:
\[ W_j=\bigoplus_{i:\,a_i=a_{i_j}}\v_i. \]
Clearly $\v=\oplus_j W_j$ and, moreover, $[W_j,W_k]=0$ if $j\neq k$. It then 
follows that the Lie subalgebras $\n_j$ of $\n$ defined by $\n_j=W_j\oplus 
[W_j,W_j]$ are in fact ideals of $\n$.

Next note that $T^2|_{W_j}=-a_{i_j}^2\operatorname{Id}$, while 
$T^2|_{[W_j,W_j]}=-9a_{i_j}^2\operatorname{Id}$, according to \eqref{KY1}. As a 
consequence, since $\n'=[\v,\v]$, we have that $\n'=\oplus_j [W_j,W_j]$ and 
furthermore,
\[ \n'\oplus\v=\oplus_j \n_j, \qquad \n=\a\oplus \n_1 \oplus \cdots \oplus \n_r,\]
a direct product of ideals. In order to show that $\n$ carries a bi-invariant 
complex structure, set $J:\n\to\n$ as follows:
\begin{enumerate}
\item[(i)] $J|_\a$ is any almost complex structure $J_0$ on $\a$ compatible 
with $\pint$;
 \item[(ii)] $J|_{W_j}=\frac{1}{a_{i_j}}T|_{W_j}$;
 \item[(iii)] $J|_{[W_j,W_j]}=\frac{1}{3a_{i_j}}T|_{[W_j,W_j]}$.
\end{enumerate}
Clearly, $J^2=-\operatorname{Id}$. In order to verify that it is bi-invariant, 
we need only check that $J[x,y]=[Jx,y]$ for any $x,y\in W_j$ for some $j$. 
Indeed,
\[J[x,y]=\frac{1}{3a_{i_j}}T[x,y]=\frac{1}{3a_{i_j}}3[Tx,y]=\frac{1}{a_{i_j}}
a_{i_j}[Jx,y]=[Jx,y].\]

\smallskip

It is clear from the construction that $J$ is skew-symmetric with respect to $\pint$. 
\end{proof}

\medskip

\begin{rem}
It follows from the proof of Theorem \ref{main2} that $\n$ can be decomposed as $\n=\mathfrak{a}\times \n_1\times \cdots \times \n_r$, an orthogonal product of $T$-invariant ideals, with $\mathfrak{a}$ abelian and $\n_i$ 2-step nilpotent. As a consequence, if $\n$ is irreducible (i.e., it cannot be decomposed as an orthogonal direct product of ideals) then $\mathfrak a=0$, $r=1$, and the endomorphisms $T$ and $J$ of $\n$ can be written, according to the orthogonal decomposition $\n=\z\oplus \v$, as
\[ T=\begin{pmatrix} 0&-3a\operatorname{Id}_p& & & & \\ 3a\operatorname{Id}_p& 
0 & & & &\\ &  & 0 & -a\operatorname{Id}_q  & &\\  & &  a\operatorname{Id}_q & 0 
& & 
\end{pmatrix}, \quad   
J=\begin{pmatrix} 0&-\operatorname{Id}_p& & & & \\ \operatorname{Id}_p& 0 & & & 
&\\ &  & 0 & -\operatorname{Id}_q  & &\\  & &  \operatorname{Id}_q & 0 & & 
\end{pmatrix},
\]
in some orthonormal basis of $\z$ and $\v$, for some $a\neq 0$. Here 
$\operatorname{Id}_t$ denotes the $t\times t$ identity matrix, 
$p=\dim \z$ and $q=\dim \v$.
\end{rem}

\begin{rem}
It was proved in \cite{AG} that any left invariant Hermitian metric on a unimodular 
complex Lie group is balanced, i.e. its fundamental 2-form is co-closed. It thus follows that the Hermitian metric $g$ in Theorem \ref{main2} is 
balanced with respect to the bi-invariant complex structure $J$.
\end{rem}

\begin{rem}
It is known that if $T$ is a KY tensor on a Riemannian manifold $(M,g)$ then 
$S:=-T^2$ is a Killing tensor on $M$ (see for instance \cite{Sem,DM}). This 
means that $S$ is symmetric and $g((\nabla_X S)X,X)=0$ for any vector field $X$ 
on $M$. In contrast with the KY case, it was shown in \cite{DM} that many 2-step 
nilpotent Lie groups, not only the complex ones, admit Killing tensors for 
certain left invariant metrics.
\end{rem}

\begin{rem}
A symmetric $(1,1)$-tensor $B$ on a Riemannian manifold $(M,g)$ is called a 
\textit{Codazzi tensor} if it satisfies $(\nabla_X B)Y =  (\nabla_Y B)X $ for 
any vector fields $X,Y$ on $M$. In the case of a left invariant metric on a 
2-step nilpotent Lie group, an analogous result to Theorem \ref{bds} holds. 
Indeed, a symmetric endomorphism of the Lie algebra $\n=\v\oplus \z$ is a 
Codazzi tensor if and only if $B(\z)\subseteq \z$ and $B[x,y]=[Bx,y]$ for all 
$x,y\in\n$. However, in this case, it is easy to see that these conditions also 
imply that $B$ is parallel (compare \cite[Proposition 2]{D}). 
\end{rem}

\

\section{Space of solutions}\label{dim=1}

In this section we will show that for certain 2-step nilpotent complex Lie groups equipped with left invariant metrics, the space of Killing-Yano tensors is one-dimensional.

We will be concerned with 2-step nilpotent complex Lie groups arising from graphs. We recall briefly this construction, introduced in \cite{DaM}.

Let $\G=\G(S,E)$ be a simple, undirected graph with set of vertices $S=\{e_1,\ldots,e_n\}$ and set of edges $E=\{z_1,\ldots,z_m\}$, with $m\geq 1$. Let $\mathcal{B}=S\cup E$ and for a field $\mathbb{K}$ with characteristic different from 2, consider the Lie algebra $\n_\G^{\mathbb K}$ whose underlying $\mathbb K$-vector space has $\mathcal B$ as a basis, and the Lie bracket is given by:
\begin{align} 
& \circ  \text{for } i<j, \quad [e_i,e_j]=\begin{cases} z_k, \quad  \text{ if } 
z_k \text{ is the edge joining } e_i \text{ and } e_j,\\
              0, \qquad   \text{otherwise}. 
             \end{cases} \\
& \circ z_r \text{ is central for any } r. \nonumber
\end{align}
Clearly, $\n_\G^{\mathbb K}$ is a 2-step nilpotent Lie algebra, and it satisfies the following very strong condition:
\begin{equation}\label{property-graph}
 \text{if } [e_i,e_j]=[e_r,e_s]\neq 0 \text{ then } i=r \text{ and } j=s.
\end{equation}
Many properties of $\n_\G^{\mathbb K}$ can be deduced directly from $\G$. For instance, the center of $\n_\G^{\mathbb K}$ is spanned by $E\cup \{\text{isolated vertices}\}$, and $\n_\G^{\mathbb K}$ is indecomposable if and only if $\G$ is connected. Moreover, it was proved in \cite{Ma} that given two graphs $\G_1$ and $\G_2$, the Lie algebras $\n_{\G_1}^{\mathbb K}$ and $\n_{\G_2}^{\mathbb K}$ are isomorphic if and only if $\G_1$ and $\G_2$ are isomorphic.

\begin{example}\label{heis}
The 3-dimensional Heisenberg Lie algebra over $\mathbb K$ arises from the graph with two vertices joined by one edge. Moreover, it is clear from \eqref{property-graph} that the $(2n+1)$-dimensional Heisenberg Lie algebra arises from a graph if and only if $n=1$.
\end{example}

\

We will study next KY tensors on 2-step nilpotent complex Lie algebras arising from graphs. So, let $\G=\G(S,E)$ be a graph as before, and let $\n_\G^{\mathbb C}$ be the \textit{complex} 2-step nilpotent Lie algebra arising from $\G$. When considered as a \textit{real} Lie algebra, it will be denoted simply by $\n_\G$. It has a real basis $\mathcal{B}'=\{e_1,f_1,\ldots,e_n,f_n\}\cup\{z_1,w_1,\ldots,z_r,w_r\}$ such that its natural bi-invariant complex structure $J$ is given by $Je_i=f_i,\, Jz_i=w_i$. The Lie bracket on $\n_\G$ satisfies
\[ [e_i,e_j]=-[f_i,f_j], \quad [e_i,f_j]=[f_i,e_j], \quad \text{ for any } i,j.\]
Note that $\n_\G$ is not a real 2-step nilpotent Lie algebra arising from a graph, since it does not satisfy \eqref{property-graph}. 

Let us define an inner product $\pint$ on $\n_\G$ by declaring the basis $\mathcal B'$ to be orthonormal. If $\G$ has no isolated vertices then the center $\z$ of $\n_\G$ is $\text{span}\{z_1,w_1,\ldots,z_r,w_r\}$, and its orthogonal complement $\v$ is $\text{span}\{e_1,f_1,\ldots,e_n,f_n\}$. 

\medskip

According to Proposition \ref{main1}, $(\n_\G, \pint)$ admits KY tensors. In the next result we determine the dimension of the space of KY tensors on $(\n_\G, \pint)$ when $\G$ is connected.

\begin{teo}\label{graph}
When the graph $\G$ is connected, the space of KY tensors on $(\n_\G, \pint)$ has dimension 1.
\end{teo}

\begin{proof}
Let $T$ be a KY tensor on $(\n_\G, \pint)$. According to Theorem \ref{bds}, $T$ preserves both $\z$ and $\v$ and, moreover, 
$T[x,y]=3[Tx,y]=3[x,Ty]$ for any $x,y\in\v$. Since $\{e_1,f_1,\ldots,e_n,f_n\}$ is an orthonormal basis of $\v$, we have that
\begin{gather}\label{formulas}
 Te_i=\sum_{j=1}^n (\langle Te_i,e_j\rangle e_j+\langle Te_i,f_j\rangle f_j), \\
 Tf_i=\sum_{j=1}^n (\langle Tf_i,e_j\rangle e_j+\langle Tf_i,f_j\rangle f_j). \nonumber
\end{gather}

Let us fix two indices $i,k\in\{1,\ldots,n\}$ with $i\neq k$. We will show that 
\begin{equation} \label{Te}
\langle Te_i,e_k\rangle=\langle Te_i,f_k\rangle =0  \quad \text{and} \quad \langle Tf_i,e_k\rangle=\langle Tf_i,f_k\rangle =0.
\end{equation}
First, let us assume that $[e_i,e_k]\neq 0$. It follows from $[Te_i,e_i]=0$ that 
\[  \sum_j \langle Te_i,e_j\rangle [e_j,e_i]+\sum_j \langle Te_i,f_j\rangle [f_j,e_i]=0. \]
The first sum in the left-hand side lies in $\text{span}\{z_1,\ldots,z_r\}$, whereas the second sum lies in $\text{span}\{w_1,\ldots,w_r\}$. Therefore, each sum vanishes and, taking \eqref{property-graph} into account, we have that $\langle Te_i,e_k\rangle [e_k,e_i]=0$ and $\langle Te_i,f_k\rangle [f_k,e_i]=0$, which imply $\langle Te_i,e_k\rangle=\langle Te_i,f_k\rangle=0$. Beginning with $[Tf_i,f_i]=0$, the same computation gives $\langle Tf_i,e_k\rangle=\langle Tf_i,f_k\rangle=0$.

\medskip 

Let us now assume that $[e_i,e_k]=0$. There exists $l\neq i$ such that $[e_k,e_l]\neq 0$, since otherwise $e_k$ would be an isolated vertex of $\G$. Since $[Te_i,e_l]=[e_i,Te_l]$, and using \eqref{formulas}, we obtain
\[ \sum_{j=1}^n \langle Te_i,e_j\rangle [e_j,e_l]+\sum_{j=1}^n\langle Te_i,f_j\rangle [f_j,e_l]=\sum_{j=1}^n \langle Te_l,e_j\rangle [e_i,e_j]+\sum_{j=1}^n\langle Te_l,f_j\rangle [e_i,f_j]. \]
The first sum in each side lies in $\text{span}\{z_1,\ldots,z_r\}$, whereas the second sum in each side lies in $\text{span}\{w_1,\ldots,w_r\}$. Therefore we obtain that
\begin{gather}
 \sum_{j=1}^n \langle Te_i,e_j\rangle [e_j,e_l]=\sum_{j=1}^n \langle Te_l,e_j\rangle [e_i,e_j], \label{first}\\
 \sum_{j=1}^n\langle Te_i,f_j\rangle [f_j,e_l]=\sum_{j=1}^n\langle Te_l,f_j\rangle [e_i,f_j].\label{second}
\end{gather}
The non-zero element $[e_k,e_l]$ appears in the left-hand side of \eqref{first}, but if it appeared also in the right-hand side, according to \eqref{property-graph} we would have that $\{k,l\}=\{i,j\}$ for some $j$, and this is impossible since both $k$ and $l$ are different from $i$. It follows that $\langle Te_i,e_k\rangle=0$. The same reasoning can be applied in \eqref{second} (since also $[f_k,e_l]\neq 0$), hence we obtain $\langle Te_i,f_k\rangle=0$.

Analogously, beginning with $[Tf_i,f_l]=[f_i,Tf_l]$, we prove that $\langle Tf_i,e_k\rangle=\langle Tf_i,f_k\rangle =0$.

\medskip

Thus, \eqref{Te} holds and since $T$ is skew-symmetric we have that $Te_i=a_if_i$ and $Tf_i=-a_ie_i$ for some $a_i\in\R$, $i=1,\ldots,n$. We will show next that $a_i=a_j$ for all $i,j$. Let us assume first that $[e_i,e_j]=z_k$. It follows from \eqref{KY1} that $T[e_i,e_j]=3[Te_i,e_j]=3[e_i,Te_j]$ and therefore we obtain that 
\[ Tz_k=3a_iw_k=3a_jw_k. \]
Therefore, $a_i=a_j$. If, on the other hand, $[e_i,e_j]=0$, we can choose a sequence of indices $i=l_0<l_1<\cdots <l_{s-1}<l_s=j$ such that $[e_{l_t},e_{l_{t+1}}]\neq 0$, since $\G$ is connected. As a consequence, we have that $a_i=a_{l_1}=\cdots=a_{l_{s-1}} =a_j$. Setting $a:=a_i$ for any $i$, the KY tensor $T$ is given by 
\[ Te_i=af_i, \quad Tf_i=-ae_i, \quad Tz_k=3aw_k,\quad Tw_k=-3az_k, \quad i=1,\ldots,n,\quad k=1,\ldots, m.\]
Hence the space of KY tensors on $(\n_\G, \pint)$ has dimension 1. 
\end{proof}

\begin{cor}
 When $\G$ is connected, any non-zero Killing-Yano tensor on $(\n_\G, \pint)$ is invertible, thus non-parallel.
\end{cor}

\medskip

Note that in the proof of Theorem \ref{graph}, the main properties used were the skew-symmetry of $T$ and equation \eqref{KY1}. The constant appearing in this formula is not relevant for the proof, and as a consequence we obtain the following result (compare with \cite[Section 5.1]{De}).

\begin{prop}
If $I$ is a bi-invariant complex structure on $\n_\G$ which is Hermitian with respect to $\pint$, then $I=\pm J$. 
\end{prop}

We learnt recently that V. del Barco and A. Moroianu obtained in \cite{DM1} the uniqueness, up to sign, of the orthogonal bi-invariant complex structure on a 2-step nilpotent metric Lie group which is de Rham irreducible. 

\medskip

\begin{rem}
As mentioned in Example \ref{heis}, the only complex Heisenberg Lie algebra $\h_{2n+1}(\C)$ which can be obtained from a graph is $\h_3(\C)$. However, it is easy to see that for $n\geq 2$, $\h_{2n+1}(\C)$ with its canonical $H$-type Hermitian metric has a space of KY tensors of dimension 1.
\end{rem}

%
%

\

\

\end{document}